\documentclass[11pt]{article}
\usepackage{amssymb,amsmath,amsthm}
\usepackage{epsfig}
\usepackage{url}

\makeatletter
\let\@@citation@@=\citation
\renewcommand{\citation}[1]{\@@citation@@{#1}%
\@for\@tempa:=#1\do{\@ifundefined{cit@\@tempa}%
  {\global\@namedef{cit@\@tempa}{}}{}}%
}
\makeatother

\usepackage{hyperref}
\hypersetup{pdfpagemode=UseNone} 

\makeatletter
\def\@lbibitem[#1]#2#3\par{%
  \@ifundefined{cit@#2}{}{\@skiphyperreftrue
  \H@item[%
    \ifx\Hy@raisedlink\@empty
      \hyper@anchorstart{cite.#2\@extra@b@citeb}%
        \@BIBLABEL{#1}%
      \hyper@anchorend
    \else
      \Hy@raisedlink{%
        \hyper@anchorstart{cite.#2\@extra@b@citeb}\hyper@anchorend
      }%
      \@BIBLABEL{#1}%
    \fi
    \hfill
  ]%
  \@skiphyperreffalse}%
  \if@filesw
    \begingroup
      \let\protect\noexpand
      \immediate\write\@auxout{%
        \string\bibcite{#2}{#1}%
      }%
    \endgroup
  \fi
  \ignorespaces
  \@ifundefined{cit@#2}{}{#3}}

\def\@bibitem#1#2\par{%
  \@ifundefined{cit@#1}{}{\@skiphyperreftrue\H@item\@skiphyperreffalse
  \Hy@raisedlink{%
    \hyper@anchorstart{cite.#1\@extra@b@citeb}\relax\hyper@anchorend
    }}%
  \if@filesw
    \begingroup
      \let\protect\noexpand
      \immediate\write\@auxout{%
        \string\bibcite{#1}{\the\value{\@listctr}}%
      }%
    \endgroup
  \fi
  \ignorespaces
  \@ifundefined{cit@#1}{}{#2}}
\makeatother

\newtheorem*{thm*}{Theorem}

\usepackage{indentfirst}
\usepackage{xspace} 

\def\B{\mathcal B}

\mathcode`l="8000
\begingroup
\makeatletter
\lccode`\~=`\l
\DeclareMathSymbol{\lsb@l}{\mathalpha}{letters}{`l}
\lowercase{\gdef~{\ifnum\the\mathgroup=\m@ne \ell \else \lsb@l \fi}}%
\endgroup

\begin{document}

\title{Exponential lower bound for Berge-Ramsey problems}
\author{D\"om\"ot\"or P\'alv\"olgyi
\footnote{MTA-ELTE Lend\"ulet Combinatorial Geometry Research Group, Institute of Mathematics, E\"otv\"os Lor\'and University (ELTE), Budapest, Hungary}
}
\maketitle



Gerbner and Palmer \cite{GP}, generalizing the definition of hypergraph cycles due to Berge, introduced the following notion.
A hypergraph $H$ contains a \emph{Berge copy} of a graph $G$, if there are injections $\Psi_1: V(G)\to V(H)$ and $\Psi_2: E(G) \to E(H)$ such that for every edge $uv\in E(G)$ the containment $\Psi_1(u),\Psi_1(v)\in \Psi_2(uv)$ holds, i.e., each graph edge can be mapped into a distinct hyperedge containing it to create a copy of $G$.
If $|E(H)|=|E(G)|$, then we say that $H$ is a \emph{Berge-$G$}, and we denote such hypergraphs by $\B G$.

The study of Ramsey problems for such hypergraphs started independently in 2018 by three groups of authors \cite{AG,GMOV,STWZ}.
Denote by $R_r(\B G; c)$ the size of the smallest $N$ such that no matter how we $c$-color the $r$-edges of $K_N^r$, the complete $r$-uniform hypergraph, we can always find a monochromatic $\B G$.
In \cite{AG} $R_r(\B K_n; c)$ was studied for $n=3,4$.
In \cite{GMOV} it was conjectured that $R_r(\B K_n; c)$ is bounded by a polynomial of $n$ (depending on $r$ and $c$), and they showed that $R_r(\B K_n; c)=n$ if $r>2c$ and $R_r(\B K_n; c)=n+1$ if $r=2c$, while $R_3(\B K_n; 2)< 2n$ (also proved in \cite{STWZ}).
In \cite{STWZ} a superlinear lower bound was shown for $r=c=3$ and for every other $r$ for large enough $c$.
This was improved in \cite{G} to $R_{r}(\B K_n; c)=\Omega(n^{d})$ if $c>(d-1)\binom r2$ and $R_r(\B K_n; c)=\Omega(n^{1+1/(r-2)}/\log n)$.
We further improve these to disprove the conjecture of \cite{GMOV}. 

\begin{thm*}
	$R_{r}(\B K_n; c)> (1+\frac 1{r^2})^{n-1}$ if $c>\binom r2$.
\end{thm*}
\begin{proof}
	It is enough to prove the statement for $c=\binom r2+1$.
	For $r=2$ this reduces to the classical Ramsey's theorem, so we can assume $r\ge 3$.
	We can also suppose $n\ge \binom r2+1=c$, or the lower bound becomes trivial.
	Suppose $N\le (1+\frac 1{r^2})^{n-1}$.
	Assign randomly (uniformly and independently) a forbidden color to every pair of vertices in $K_N^r$.
	Color the $r$-edges of $K_N^r$ arbitrarily, respecting the following rule:
	if $\{u,v\}\subset E$, then the color of $E$ cannot be the forbidden color of $\{u,v\}$.
	Since $c>\binom r2$, this leaves at least one choice for each edge.
	Following the classic proof of the lower bound of the Ramsey's theorem, now we calculate the probability of having a monochromatic $\B K_n$.
	The chance of a monochromatic $\B K_n$ on a fixed set of $n$ vertices for a fixed color is at most $(\frac {c-1}c)^{\binom n2}$, as the fixed color cannot be the forbidden one on any of the pairs of vertices.
	Thus the expected number of monochromatic $\B K_n$'s is at most
	$c\binom Nn(\frac {c-1}c)^{\binom n2}$.
	If this quantity is less than $1$, then we know that a suitable coloring exists.
	Since $c\le n\le n!$, it is enough to show that $N< (\frac c{c-1})^{\frac {n-1}2}$, but this is true using $c=\binom r2+1$ and $r\ge 3$. 
\end{proof}

\subsection*{Remarks and acknowledgment}
As was brought to my attention by an anonymous referee, my construction for $r=3$ and $c=4$ is essentially the same as the one used in the proof of Theorem 1(ii) in \cite{CFR} for a different problem, the $4$-color Ramsey number of the so-called \emph{hedgehog}.
A hedgehog with body of order $n$ is a $3$-uniform hypergraph on $n+\binom n2$ vertices such that $n$ vertices form its body, and any pair of vertices from its body are contained in exactly one hyperedge, whose third vertex is one of the other $\binom n2$ vertices, a different one for each hypderedge.
It is easy to see that such a hypergraph is a Berge copy of $K_n$, and while their result, an exponential lower bound for the $4$-color Ramsey number of the hedgehog, does not directly imply mine, their construction is such that it also avoids a monochromatic $\B K_n$.

It is an interesting problem to determine how $R_{r}(\B K_n; c)$ behaves if $c\le \binom r2$.
The first open case is $r=c=3$, just like for hedgehogs.

\end{document}